\documentclass{amsart}
 \usepackage{graphicx}
\newtheorem{theorem}{Theorem}[section]
\newtheorem{lemma}[theorem]{Lemma}
\newtheorem{corollary}[theorem]{Corollary}
\newtheorem{proposition}[theorem]{Proposition}
\newtheorem{algorithm}[theorem]{Algorithm}
\theoremstyle{definition}

\theoremstyle{remark}
\newtheorem{remark}[theorem]{Remark}
\numberwithin{equation}{section}

\newcommand{\ml}{\mathcal L}
\newcommand{\lip}{\text{Lip}}
\newcommand{\loc}{\text{loc}}

\newcommand{\proj}{\text{Proj}}

\begin{document}
\title{Invariant densities and escape rates:\\ Rigorous and computable approximations in the $L^{\infty}$-norm}
\author{Wael Bahsoun}
    \address{Department of Mathematical Sciences, Loughborough University, 
Loughborough, Leicestershire, LE11 3TU, UK}
\email{W.Bahsoun@lboro.ac.uk}
\author{Christopher Bose}
\address{Department of Mathematics and Statistics, University of Victoria,  
   PO BOX 3045 STN CSC, Victoria, B.C., V8W 3R4, Canada}
\email{cbose@uvic.ca}
\date{\today}

\thanks{CB is supported by an NSERC grant.  WB acknowledges the hospitality of the 
Department of Mathematics and Statistics, University of Victoria, where much of this research was 
carried out}
\begin{abstract}
In this article we study a piecewise linear discretization schemes for transfer operators (Perron-Frobenius operators) associated with interval maps. We show how these can be used to  provide rigorous {\bf pointwise} approximations for invariant densities of Markov interval maps. We also derive the order of convergence of the approximate invariant density to the real one in the $L^{\infty}$-norm. The outcome of this paper complements rigorous results on $L^1$ approximations of invariant densities \cite{KMY} and recent results on the formulae of escape rates of open dynamical systems \cite{KL2}. We implement our computations on two examples (one rigorous and one non-rigorous) to illustrate the feasibility and efficiency of our 
schemes.\end{abstract}
\maketitle
\pagestyle{myheadings} 
\markboth{Rigorous Approximations In The $L^{\infty}$-Norm}{W. Bahsoun And C. Bose}
\section{Introduction}
Although this article is about the approximation of invariant densities for interval maps, it is intimately  related to what are commonly termed \textit{open dynamical systems} or maps with `holes' \cite{DY}. Open dynamical systems have become a very active area of research. In part, this is due to their connection to metastable dynamical systems \cite{FS, GHW} and their applications in earth and ocean sciences \cite{DF, SFM}. 
Corresponding to invariant measures for closed dynamics,  in open dynamical systems, long-term statistics are described by a \textit{conditionally invariant measure} and its related \textit{escape rate}, measuring the mass lost from the system per unit time \cite{DY}. 

\bigskip

In their recent article \cite{KL2}, Keller and Liverani obtained precise escape rate formulae for Lasota-Yorke maps with holes shrinking to a single point. These formulae depend, \textit{pointwise}, on the invariant density of the corresponding closed system. Unfortunately, explicit formulae of invariant densities for Lasota-Yorke maps are, in general, unavailable. Thus, to complement the result of \cite{KL2}, it is natural to consider numerical schemes which provide rigorous and computable  pointwise approximations
of invariant densities.

\bigskip

In the literature, rigorous approximation results are available in the $L^1$-norm \cite{KMY},
the $L^{\infty}$-norm \cite{BIS} and in the $BV$-norm, the space of function of bounded variation, \cite{Ha,DL}, for example. None of these methods 
are well-suited to our problem.  For example $L^1$ approximations cannot 
provide the pointwise information necessary for application of the formulae of \cite{KL2} (see Section 7 in \cite{BB}). Convergence schemes in the $L^{\infty}$-norm and $BV$-norm cannot provide a computable error bound as the constant in their approximation error depends on the invariant density itself, which is 
\textit{a priori \/} unknown. 

\bigskip

By rigorous and computable approximation we mean the following. Assume we are given a transformation
$\tau$ (typically, a formula) and an error tolerance,  for example $\Delta := 10^{-2}$. We choose a suitable discretization 
scheme for the transfer (Frobenius-Perron) operator associated to $\tau$ and we are asked to determine an explicit 
level of discretization $\epsilon$ such that the approximate invariant density $f^*_\epsilon$  for the discretize operator at level $\epsilon$ satisfies
\begin{equation}\label{eqn:approx}
\| f^* - f^*_\epsilon\|_\infty \leq \Delta
\end{equation}
Here $f^*$ is the invariant density for $\tau$. 
We emphasize that we assume the continuous density $f^*$ is unknown throughout this 
calculation.   By computable we mean that, at each step, one can determine via an algorithm, 
within a finite number of steps, each quantity necessary to determine\footnote{Of course, the efficiency of such an algorithm is an important issue. For our purpose, we will be satisfied with algorithms that can be implemented in standard mathematical software on a personal computer. Beyond that, we do not specifically address computational efficiency in this article} $f^*_\epsilon$ and to guarantee Inequality (\ref{eqn:approx}).

One novelty of our approach is that we will need to consider two different discretization schemes in order 
to carry out this task, both based on binned discretization of the state space.   
The results in this paper enable us to compute the number of bins $m$, with  $\epsilon=m^{-1}$, and the associated
approximate
density, as usual denoted  $f^*_m$ which achieves the tolerance $\Delta$, uniformly.

\bigskip 

The first ingredient in our analysis uses a natural piecewise linear discretization scheme and the abstract perturbation result of \cite{KL}. The two Banach spaces involved in our computation are $L^{\infty}$, and the space of Lipschitz continuous functions on the unit interval. The same Banach spaces were used in \cite{Fr} to provide a computer-assisted estimate on the rate of decay of correlations. However, the discretization scheme which was used in \cite{Fr} is the traditional Ulam method.  Ulam's method does not fit in our setting since it does not preserve the regularity of the space of Lipschitz continuous functions. The idea of our method is to first  prove an appropriate Lasota-Yorke inequality for transfer operators associated with Markov interval maps,  then to construct a discretized transfer operator which preserves the regularity of the space of Lipschitz continuous functions and which is close, in some suitable norm, to the original transfer operator. Although, neither the original transfer operator nor its discretized counterpart is a contraction in the $L^{\infty}$-norm\footnote{Unlike the $L^1$ setting where the norm of the transfer operator is automatically $||\ml ||_1\le 1$, here we can only show that $|\ml|\le M$, where $M$ is typically greater than or equal to 1.}, we obtain quasi-compactness of the original transfer operator, thanks to the result of \cite{He}. We use the general setting of \cite{KL} which allows the study of perturbations of transfer operators which are not necessarily contractions on either  Banach space. With this machinery, we can compute an explicit upper bound on the norm of the resolvent, bounded away from the spectrum, of the transfer operator associated with the map. Once an upper bound on the norm of the resolvent is computed, we use a second discretization scheme, whose associated finite rank operator is Markov, to compute an approximate invariant density with the pre-specified error tolerance $\Delta$. 

\bigskip

The reason for using two different discretizations in our method is the following: The first discretization has the projection property. This property is essential in the proofs related to the computation of  an an explicit upper bound on the norm of the resolvent. Moreover, it produces reasonable constants\footnote{This is very important from computational point of view. In particular, smaller constants means that less time will be spent on the computer to produce the desired computation.} which are needed when using the perturbation result of \cite{KL}. However, this natural discretization leads to a non-Markovian finite rank operator. The lack of the Markov property makes the (theoretical) rate of convergence slow. Thus, at
the next stage, we use a different discretization, which lacks the projection property\footnote{Thus, we could not use the Markov discretization right from the beginning.}, but produces a finite rank operator which is Markov. With this Markov scheme we will obtain a computable rate of convergence which is of order $m^{-1} \, \ln m$.  

\bigskip

In Section \ref{Spaces} we set up our notation and assumptions. We also recall known results on Markov interval maps which are needed in the sequel.  In Section \ref{Inequality} we provide a Lasota-Yorke inequality on the space of Lipschitz continuous functions. In Section \ref{Discretization} we present our two discretization schemes and prove results about their regularity properties when acting on the space of Lipschitz continuous functions. In Section \ref{Perturbation} we present the perturbation result of \cite{KL} as a sequence of steps which are necessary for rigorous computations and in Section \ref{sec:KL_applied} we apply the perturbation result to our problem. The main challenge of this paper lies in this section where we design an algorithm which enables one to rigorously compute an upper bound on the norm of the resolvent of the continuous transfer operator. The resolvent estimate can then be used  to obtain a rate of convergence $Cm^{-1} \, \ln m$ in Section \ref{Rates},  where $C$ is a computable constant. In Section \ref{Exam} we implement the algorithm of Section \ref{sec:KL_applied} on a specific Markov interval map. Numerical results are reported for all critical constants. In particular, we compute the number 
of bins ($m=7 \times 10^6$) that guarantee approximation of $f^*$ by $f_m^*$ within tolerance $\Delta = 10^{-2}$. Section \ref{discuss} contains a discussion of an alternate projection based (non Markov) scheme and in general,  on efficiency of both piecewise linear discretization schemes for uniform 
approximation. 

\section{Assumptions and Notation}\label{Spaces}
\subsection{Function spaces}
Let $(I,\mathfrak B, \lambda)$ denote the measure space where $I:=[0,1]$, $\mathfrak B$ is Borel $\sigma$-algebra and $\lambda$ is Lebesuge measure on $I$. Let $L^p= L^{p}(I)=L^{p}(I,\mathfrak B,\lambda)$
for $1\leq p \leq \infty$.  Let $\mathcal C^{\lip}(I)$ denote the space of Lipschitz continuous functions on $I$. We equip $\mathcal C^{\lip}(I)$ with the norm
$$||\cdot||=\lip(\cdot)+|\cdot|,$$ 
where $\lip(\cdot)$ is the Lipschitz constant of a function and $|\cdot|$ is its $L^{\infty}$ norm. The Arzel\`a-Ascoli theorem implies that the unit ball of $||\cdot||$ is $|\cdot|$-compact\footnote{See \cite{DS} for this and other basic properties about the Banach space $\mathcal C^{\lip}(I)$.} 
\subsection{Markov maps of the unit interval}
Let $\tau:I\to I$ be a measurable transformation. We assume that there exists a partition $\mathcal P $ of $I$, $\mathcal P=\{c_i\}_{i=0}^q$ such that
\begin{enumerate}
\item for each $i=1,\dots, q$, $\tau_i:=\tau_{|(c_{i-1}, c_i)}$ is monotone, $C^{1+\lip}$ and it extends to a $C^1$ function on $[c_{i-1}, c_i]$;
\item $(\tau_i^{-1})':[\tau(c_{i-1}),\tau(c_i)]\to[c_{i-1},c_i]$ is a Lipschitz continuous function for each $i$.
\item there exits a number $\alpha_0$ such that $\frac{1}{|\tau_i'|}\le\alpha_0<1$;
\item for each $i,j=1,\dots,q$, if $\tau((c_{i-1}, c_i))\cap(c_{j-1}, c_j)\not=\emptyset$, then $\tau((c_{i-1}, c_i))\supseteq(c_{j-1}, c_j);$
\item there exists a $k_0\ge 1$ such that $\tau^{k_0}$ is piecewise onto.
\end{enumerate}
Let $\ml_{\tau}:L^1\to L^1$ denote the transfer operator (Perron-Frobenius) \cite{Ba} associated to $\tau$:
$$\ml_{\tau} f(x)=\sum_{y=\tau^{-1}x}\frac{f(y)}{|\tau'(y)|}.$$

\subsection{Results implied by the assumptions on $\tau$} Assumptions (1)-(5) imply (see \cite{Bo}) that
$\ml_\tau:\mathcal C^{\lip}_{\loc}(I) \rightarrow \mathcal C^{\lip}_{\loc}(I)$, the space of locally 
Lipschitz functions (relative to the Markov partition $\mathcal{P}$) and 
$\ml_{\tau^k}:C^\lip(I) \rightarrow C^\lip(I)$ whenever $k \geq k_0$. Furthermore,  $\tau$ 
admits a unique absolutely continuous probability measure $d\nu = f^* d\lambda$ equivalent to Lebesgue, 
with $f^* \in \mathcal C^{\lip}(I)$ and uniformly bounded away from zero (\cite{Bo} Theorem 0.1). 
Further properties of the system $(\tau , \nu)$ are as follows:
\begin{enumerate}
\item The system is Bernoulli (\cite{Bo} Theorem 0.3) hence $\tau$ is mixing with 
respect to $\nu$. 
\item The eigenvalue $1$ is a simple eigenvalue for the operator $\ml_\tau$ 
acting on either of the Banach spaces $L^2(I, \nu)$ or $\mathcal C^{\lip}_{\loc}(I)$.
\item  $1$ is the only eigenvalue of unit modulus for the operator $\ml_\tau$ 
acting on either Banach space.
\end{enumerate}
A useful regularity estimate playing a critical role in the proof 
of Theorem 0.1 in \cite{Bo} is the following:
\begin{equation}\label{eq:lip_bound}
B_0:= \sup_{n \geq k_0} \lip( \ml_\tau^n \mathbf{1}) < \infty
\end{equation}
\section{Lasota-Yorke inequality on $\mathcal C^{\lip}(I)$}\label{Inequality}
\begin{lemma}\label{Le1}
For all $n\ge 1$
$$|\ml_{\tau}^n|\le M,$$
where $M=B_0+1$.
\end{lemma}
\begin{proof}
Since each function $\ml^n \mathbf{1}$ is nonnegative on $I$ with 
unit $L^1-$norm, it follows
that $| \ml^n \mathbf{1}| \leq 1 + \lip(\ml^n \mathbf{1})$.
The result now follows from  (\ref{eq:lip_bound}). 
\end{proof}
We now fix a $k\ge k_0$ such that $\alpha=M\alpha^k_0<1$. Let $T:=\tau^k$ and $\ml=\ml_\tau^k$ be the transfer operator associated to $T$. We also set $\mathcal P_0=\{a_i\}_{i=0}^{l}$ to be the partition of $I$ which consists of the intervals of monotonicity of $T$.
\begin{remark}
By the mixing property of $(\tau, \nu)$
 $\tau$ and $T$ have the same unique invariant density $f^*$. While our goal is to provide a rigorous approximation of the density $f^*$ for $\tau$ in the $L^{\infty}$ norm, we will achieve this goal by working with $T$  instead of $\tau$.
\end{remark}
\begin{lemma}\label{Le2} We have the following 
Lasota-Yorke-type inequality for  $\ml: \mathcal C^{\lip}(I) \rightarrow  \mathcal C^{\lip}(I)$ and
$$\lip(\ml f)\le\alpha\lip(f)+B_1|f|,$$
where $B_1:= \lip(\ml_\tau^k \mathbf{1})~(=\lip(\ml \mathbf{1}))$.
\end{lemma}
\begin{proof}
Let $x,y\in I$. For $f\in\mathcal C^{\lip}(I)$ we have
\begin{equation*}
\begin{split}
|\ml f(x)-\ml f(y)|&=|\sum_{i=1}^{l}\frac{f(T_i^{-1}x)}{|T'(T_i^{-1}x)|}-\sum_{i=1}^{l}\frac{f(T_i^{-1}y)}{|T'(T_i^{-1}y)|}|\\
&\le\sum_{i=1}^{l}|\frac{f(T_i^{-1}x)-f(T_i^{-1}y)}{|T'(T_i^{-1}x)|}|+\sum_{i=1}^{l}|\frac{f(T_i^{-1}y)}{|T'(T_i^{-1}x)|}-\frac{f(T_i^{-1}y)}{|T'(T_i^{-1}y)|}|\\
&\le\lip(f)\sum_{i=1}^{l}|\frac{T_i^{-1}(x)-T_i^{-1}(y)}{|T'(T_i^{-1}x)|}|+|f|\sum_{i=1}^{l}|(T_i^{-1})'(x)-(T_i^{-1})'(y)|\\
&\le\alpha_0^k\sum_{i=1}^{l}\frac{1}{|T'(T_i^{-1}x)|}\lip(f)|x-y|+\lip(\ml \mathbf{1})|f||x-y|\\
&\le\alpha\lip(f)|x-y|+B_1|f||x-y|.
\end{split}
\end{equation*}
\end{proof}
\begin{corollary}\label{Cor0}
For $f\in\mathcal C^{\lip}(I)$ and for all $n\ge1$, we have 
$$||\ml^n f||\le\alpha^n||f||+M(1+\frac{B_1}{1-\alpha})|f|.$$
\end{corollary}
\begin{proof}
The proof follows from Lemmas \ref{Le2} and \ref{Le1}.
\end{proof}
\begin{remark}\label{rem:simple_spectral_pic}
The inequality in Corollary \ref{Cor0}, combined with the mixing 
property of $T$ gives the following spectral picture for $\ml$ acting on $\mathcal C^{\lip}(I)$:
\begin{enumerate}
\item The essential spectral radius\footnote{See, for example Hennion \cite{He} for an elegant proof 
of this fact.} of $\ml$ is $\rho_{\textnormal{ess}} \leq \alpha < 1$
\item The spectrum of $\ml$ outside of the disc $\{\rho \leq \alpha\}$ consists of 
a simple eigenvalue at $\rho = 1$ corresponding to the unique invariant density 
for $T$ (or $\tau$). 
\end{enumerate}
This picture will play a critical role in what follows. 
\end{remark}
\subsection{Estimates on the relevant constants}
The key expression from the previous section is the inequality in Corollary \ref{Cor0}. 
Reviewing the constants that need to be estimated, 
$\alpha$ and $B_1$ are easily computed directly from 
the definition of $\tau$, but the constant $M$, which is defined in 
terms of $B_0$ may not
be so straightforward.
One favourable situation is when
a variational Lasota-Yorke inequality is available for $\tau$. In that case we can 
estimate $M$ directly as follows. We assume $\ml_\tau$ satisfies
\begin{equation}\label{eq:variation_estimate}
V \ml_\tau^n f \leq A_0 \alpha^n V f + B \| f \|, ~~n \geq 1
\end{equation}
where $V(\cdot)$ denotes the usual notion of variation for $L^1$ functions
and $\| \cdot \|$ the $L^1-$ norm.
Then we simply observe
\begin{equation*}
|\ml_{\tau}^n| \leq |\ml_\tau^n \mathbf{1}| \leq V( \ml_\tau^n \mathbf{1}) + \|  \ml_\tau^n \mathbf{1}\| \leq
 A_0 \alpha^n V \mathbf{1} + B \| \mathbf{1}\| + 1 = B + 1
\end{equation*}
so we can set $M:= B+1$.


\section{Two piecewise linear discretization schemes for $\ml$}\label{Discretization}
\subsection{Projection-Based Discretization}
Let $\eta=\{b_{i}\}_{i=0}^{m}$ be a partition of $I$ into intervals of size at most $\varepsilon$; i.e., $\text{mesh}(\eta)\le\varepsilon$. Since uniform partitions are the first choice for numerical work, we will 
assume for the rest of this article that $b_i - b_{i-1} = \frac{1}{m}$.  Everything we do could be modified 
for non-uniform partitions with only minor notational changes.  Let 
$$\varphi_{i}=\chi_{[b_{i-1},b_{i}]}, z_{i}=[\int_{0}^1\varphi_{i}d\lambda]^{-1}=\frac{1}{b_i-b_{i-1}}=m \text{ and } \phi_{i}(x)=m\int_0^{x}\varphi_{i}d\lambda.$$
For $f: I \rightarrow \mathbb{R}$ define  
\begin{equation*}
\begin{split}
\Pi_mf=&f(b_{0})+\sum_{i=1}^{m}\phi_{i}[f(b_{i})-f(b_{i-1})]\\
&=f(b_{0})(1-\phi_{1})+f(b_{m})\phi_{m}+\sum_{i=1}^{m-1}f(b_{i})(\phi_{i}-\phi_{i}).
\end{split}
\end{equation*}
Thus, $\Pi_mf$ is a piecewise linear function with respect to $\eta$; moreover, $\Pi_mf(b_i)=f(b_i)$, $i=0,\dots,m$. We shall write $\Pi_m$ in a more compact form
$$\Pi_mf=\sum_{i=0}^{m}f(b_{i})\psi_{i},$$
where the basis $\psi_i$ is a set of {\bf hat} functions over $\eta$:
\begin{equation}\label{hat}
\psi_{0}:=(1-\phi_{1})\,, \psi_{m}:=\phi_{m}\text{ and for }i=1,\dots,m-1\,, \psi_{i}:=(\phi_{i}-\phi_{i+1}).
\end{equation}
\begin{lemma}\label{Le3} The operator $\Pi_m$ is a projection; i.e., 
$$\Pi_m^2=\Pi_m.$$
\end{lemma}
\begin{proof}
Observe that
\begin{equation*}
\psi_i(b_j)=\left\{\begin{array}{cc}
1&\mbox{if $i=j$}\\
0&\mbox{if $i\not=j$}
\end{array}
\right. .
\end{equation*}
For $f\in\mathcal C^{\lip}(I)$,  we have
$$\Pi_m(\Pi_mf)=\sum_{j=0}^{m}\left(\sum_{i=0}^{m}f(b_i)\psi_i(b_j)\right)\psi_j=\sum_{j=0}^{m}f(b_j)\psi_j=\Pi_m.$$
\end{proof}
\begin{remark}
The projection property of $\Pi_m$ is going to be essential for results in Section \ref{sec:KL_applied}
(specifically, in the proof of Lemma \ref{Le7}). \end{remark}
\begin{lemma}\label{Le4}
For $f\in\mathcal C^{\lip}(I)$ we have
\begin{enumerate}
\item $\lip(\Pi_mf)\le\lip(f)$;
\item $|\Pi_mf|\le|f|$;
\item $|\Pi_mf-f|\le2\varepsilon\lip(f)$;
\item $V\Pi_mf\le Vf$.
\end{enumerate}
\end{lemma}
\begin{proof}
For $f\in\mathcal C^{\lip}(I)$.  We have
$$\lip(\Pi_mf)=\max_{i}\frac{|\Pi_mf(b_i)-\Pi_mf(b_{i-1})|}{b_i-b_{i-1}}=\max_{i}\frac{|f(b_i)-f(b_{i-1})|}{b_i-b_{i-1}}\le\lip(f),$$
and
$$|\Pi_mf(x)|=|\sum_{i=0}^mf(b_i)\psi_i(x)|\le|f|\sum_{i=0}^{m}\psi_i(x)=|f|.$$
To prove the third statement, let $x\in[b_{i-1},b_i]$. Then
\begin{equation*}
\begin{split}
|\Pi_mf(x)-f(x)|&=|\Pi_mf(x)-\Pi_mf(b_i)+f(b_i)-f(x)|\\
&\le|\Pi_m||f(x)-f(b_i)|+|f(b_i)-f(x)|\le2|f(b_i)-f(x)|\le 2\varepsilon\lip(f).
\end{split}
\end{equation*}
Finally, we have
$$V\Pi_m f=\sum_{i=0}^{m-1}|f(b_{i+1})-f(b_i)|\le Vf.$$
\end{proof}
We now define a discretized version of $\ml$ by:
$$\ml_m=\Pi_m\ml \Pi_m.$$
Observe that $\ml_m$ is a finite rank operator whose range is contained in the space of continuous,  piecewise linear functions with respect to $\eta$.  On this space, and with 
respect to the basis $\{\psi_i\}$, the action of  $\ml_m$ becomes matrix multiplication\footnote{All our matrices act by multiplication on the right, in keeping with the usual convention in the literature.} by an $(m+1)\times(m+1)$ matrix whose $(ij)^{\text{th}}$ entry is given by
$$m_{ij}:=\ml\psi_{i}(b_{j}).$$
By definition, for all $i$, the functions $\psi_{i}$ are non-negative. Consequently, the matrix corresponding to $\ml_m$ is a non-negative matrix. However, the matrix of $\ml_m$ is not necessarily stochastic. 

\subsection{Markov Discretization}\label{sec:markov}

Another natural choice for constructing piecewise linear approximations was used by 
Ding and Li \cite{DL}. In this case the goal is slightly different, namely, to preserve the Markov
structure of the transfer operator $\ml$. 

Let $\psi_i, i = 0, 1, 2, \dots m$  be the hat functions defined in (\ref{hat}). For $f \in L^1$, 
we set $I_i := [b_{i-1}, b_i]$ and
$$f_i := m \int_{I_i} f \, dx, ~~i= 1, 2, \dots m,$$
the average of $f$ over the associated partition cell. For $f \in L^1$ we set
$$Q_m f := f_1\psi_0 + \sum_{i=1}^{m-1}  \frac{f_i + f_{i+1}}{2}  \psi_i  + f_m \psi_m$$
While $Q_m$ fails to be a projection operator, it retains good stochastic properties\footnote{See \cite{DL} for proofs.} \cite{DL}:
\begin{itemize}
\item $Q_m: L^1 \rightarrow L^1$  is a Markov operator whose range is contained in the class of continuous, piecewise linear functions with respect to the $b_i$.    
\item $V\, Q_mf \leq V\,f $
\end{itemize}
Moreover, $Q_m$ has nice regularity properties when acting of the the space $\mathcal C^{\lip}(I)$
\begin{lemma}\label{Le:new}
For $f\in\mathcal C^{\lip}(I)$, we have
\begin{enumerate}
\item $|Q_m f|\le 2|f|$;
\item $\lip(Q_mf)\le\frac32\lip(f)$;
\item $|Q_mf- f|\le\frac{5}{2m}\lip(f)$.
\end{enumerate} 
\end{lemma}
\begin{proof}
Let $f\in\mathcal C^{\lip}(I)$ and $x\in I_i$, $i=2,\dots{m-1}$. Then, 
$$Q_mf(x)=|\frac{f_{i-1}+f_i}{2}\psi_{i-1}(x)+\frac{f_{i}+f_{i+1}}{2}\psi_{i}(x)|\le 2|f|.$$
For the special cases when $x\in I_1$ or in $I_m$, the proof is similar. For the second statement we have
\begin{equation*}
\begin{split}
\lip (Q_mf)=m\max_{i}| \frac{f_i + f_{i+1}}{2}- \frac{f_{i-1} + f_{i}}{2}|&=\frac m2\max_{i}| f_{i+1}- f_{i-1}|\\
&=\frac m2\max_{i}|f(z_{i+1})-f(z_{i-1})|,
\end{split}
\end{equation*}
for some $z_{i+1}\in I_{i+1}$ and $z_{i-1}\in I_{i-1}$. Therefore,
$$\lip (Q_mf)=\frac m2\lip(f)\max_{i}|z_{i+1}-z_{i-1}|\le\frac32 \lip(f).$$
To prove the last statement, we let $x\in I_i$ and observe that
\begin{equation*}
\begin{split}
|Q_mf(x)-f(x)|&=|\frac{f_{i-1}+f_i}{2}\psi_{i-1}(x)+\frac{f_{i}+f_{i+1}}{2}\psi_{i}(x)-f(x)|\\
&=|\frac{f_i}{2}+ \frac{f_{i-1}}{2}\psi_{i-1}(x)+\frac{f_{i+1}}{2}\psi_{i}(x)-f(x)|\\
&=|\frac{f_i-f(x)}{2}+ \frac{f_{i-1}}{2}\psi_{i-1}(x)+\frac{f_{i+1}}{2}\psi_{i}(x)-(\psi_{i-1}(x)+\psi_{i}(x))\frac{f(x)}{2}|\\
&\le|\frac{f_i-f(x)}{2}|+|\frac{f_{i-1}-f(x)}{2}|+|\frac{f_{i+1}-f(x)}{2}|\\
&\le\frac{1}{2}[\frac1m\lip(f)+\frac2m\lip(f)+\frac2m\lip(f)]=\frac{5}{2m}\lip(f).
\end{split}
\end{equation*}
\end{proof}
We now define a second piecewise linear Markov discretization of $\ml$ by
$$\mathbb{P}_m:=Q_m\circ \ml.$$
Notice that $\mathbb{P}_m$ is a finite-rank Markov operator whose range is 
contained in the space of continuous, piecewise linear functions with respect to $\eta$. 
The matrix representation of $\mathbb{P}_m$ restricted to this finite-dimensional space and
with respect to the basis $\{\psi_i\}$ is a (row) stochastic matrix, with entries 
$$p_{ij}:=m \int_{I_j} \ml\psi_{i}\geq 0.$$
\section{Keller-Liverani's theorem}\label{Perturbation}
In this section we present a version\footnote{In \cite{Li} another version of \cite{KL} was designed for rigorous computations in the framework of $(BV, ||\cdot||)$ and $(L^1,|\cdot|)$, where the $|\cdot|$-norm of both operators in the setting of \cite{Li} is smaller than 1.} of the perturbation result of \cite{KL} which is designed for rigorous computations. All the constants involved in the statements below are crucial in our work. 
\subsection{Notation}
Let $(\mathbb B, ||\cdot||)$ be a Banach space which is equipped with a second norm $|\cdot |$ such that 
\begin{equation}\label{eq:norm}
|\cdot|\le||\cdot||.
\end{equation}
For any bounded linear operator $P: \mathbb B\to \mathbb B$, consider the set
$$V_{\delta,r}(P)=\{z\in\mathbb C: |z|\le r\text{ or dist}(z,\sigma(P))\le\delta\},$$
where $\sigma(P)$ is the spectrum of $P$ with respect to $(\mathbb B, ||\cdot||)$, and define
$$H_{\delta,r}(P):=\sup\{||(z-P)^{-1}||: z\in\mathbb C\setminus V_{\delta,r}\}<\infty.$$
Further, we define the following operator norm:
$$|||P|||=\underset{||f||\le 1}{\sup}|Pf|.$$
\subsection{Assumptions}
let $P_i:\mathbb B\to \mathbb B$  be two bounded linear operators, $i=1,2$. 
Assume that:
there is a $C_1, M>0$ such that for all $n\in \mathbb N$
\begin{equation}\label{E0} 
|P_i^n|\le C_1M^n;
\end{equation}
and $\exists\,\alpha\in (0,1)$, $\alpha<M$,  and $C_2, C_3>0$ such that
\begin{equation}\label{E1} 
||P_{i}^nf||\le C_2\alpha^n||f||+C_3M^n|f| \hskip .3cm \forall n\in\mathbb N\,\ \forall f\in \mathbb B,\, i=1,2;
\end{equation}
moreover, if 
\begin{equation}\label{E2}
|z|>\alpha,\text{ then } z \text{ is not is the residual spectrum of } P_i, i=1,2.
\end{equation}
\subsection{Results}
\begin{theorem}\cite{KL}\label{Th1}
Consider two operators $P_i:\mathbb B\to \mathbb B$ which satisfy (\ref{E0}), (\ref{E1}) and (\ref{E2}). For\footnote{In this paper will be interested in those $r\in(\alpha,1)$.} $r\in(\alpha, M)$, let 
$$n_1=\lceil \frac{\ln 2C_2}{\ln r/\alpha}\rceil$$
$$n_2=\lceil\frac{\ln 8C_3C_2(C_2+C_3+2)MH_{\delta,r}(P_1)}{\ln r/\alpha}+n_1\frac{\ln(M/r)}{\ln r/\alpha}\rceil .$$
If
$$|||P_1-P_2|||\le\frac{\left(r/M\right)^{n_1+n_2}}{8C_3(H_{\delta,r}(P_1)C_3+\frac{C_1}{M-r})} :=\varepsilon_1(P_1,r,\delta)$$
then for each $z\in\mathbb C\setminus V_{\delta,r}(P_1)$, we have
$$||(z-P_2)^{-1}f||\le \frac{4(C_2+C_3)}{M-r}\left (\frac{M}{r}\right )^{n_1}||f||+\frac{1}{2\varepsilon_1}||f||_1.$$
Set
$$\gamma=\frac{\ln(r/\alpha)}{\ln(M/\alpha)},$$ 
$$a=\frac{8M(C_2+C_3)^2}{M-r}\left(\frac{M}{r}\right)^{n_1}[2C_2(C_2+C_3)+\frac{C_1}{M-r}]+\frac{2C_1}{M-r},$$

$$b=\frac{8M}{M-r}\left[MC_2(C_2+C_3+2)+C_3\right](C_2+C_3)^2\left(\frac{M}{r}\right)^{n_1}+2C_3,$$
and
\begin{equation*}
\begin{split}
&\varepsilon_2(P_1,r,\delta):=\\
&\left[\frac{1}{4C_3}\left(\frac{r}{M}\right)^{n_1}\left( \frac{1}{H_{\delta,r}(P_1)[C_2(C_2+C_3+2)M+C_3]+2C_2(C_2+C_3)+\frac{C_1}{M-r}}\right)\right]^{\frac{1}{\gamma}}.
\end{split}
\end{equation*}
If
\begin{equation}\label{E3}
|||P_1-P_2|||\le\min\{\varepsilon_1(P_1,r,\delta),\varepsilon_2(P_1,r,\delta)\}:=\varepsilon_0(P_1,r,\delta)
\end{equation}
then for each $z\in\mathbb C\setminus V_{\delta,r}(P_1)$, we have
\begin{equation}\label{E4}
|||(z-P_2)^{-1}-(z-P_{1})^{-1}|||\le |||P_1-P_2|||^{\gamma}(a||(z-P_1)^{-1}||+b||(z-P_1)^{-1}||^2). 
\end{equation}
\end{theorem} 
\begin{corollary}\cite{KL}\label{Co1}
If $|||P_1-P_2|||\le\varepsilon_1(P_1,r,\delta)$ then $\sigma(P_2)\subset V_{\delta,r}(P_1)$. In addition, if 
$|||P_1-P_2|||\le\varepsilon_0(P_1,r,\delta)$, then in each connected component of $V_{\delta,r}(P_1)$ that does not contain $0$ both 
$\sigma(P_1)$ and $\sigma(P_2)$ have the same multiplicity; i.e., the associated spectral projections have the 
same rank.
\end{corollary}
\section{Computing the decay rate for $\ml$}\label{sec:KL_applied}
\subsection{Applying Keller-Liverani's theorem to $\ml$ and $\ml_m$}
\begin{lemma}\label{Le5}
Let $\ml_{*}$ denote either $\ml$ or $\ml_m$ and let $f\in\mathcal C^{\lip}(I)$. We have
\begin{enumerate}
\item for all $n\ge 1$, $|\ml^n_{*}|\le M^n$, where $M=B_0+1$;
\item for all $n\ge 1$, $||\ml^n_{*}f||\le\alpha^{n}||f||+C_3M^n|f|,$
where $C_3=\frac{B_1}{M(1-\alpha)}+1$.
\item $|||(\ml-\ml_m)||| \le\Gamma\varepsilon$, where $\Gamma=2\cdot\max\{\alpha+M,B_1\}$. 
\end{enumerate}
\end{lemma}
\begin{proof}
The inequalities in the first two statements for $\ml$ are dominated by those of $\ml_m$. Therefore we prove them only for $\ml_{*}=\ml_m$. For $|f|\le1$, using Lemmas \ref{Le1} and \ref{Le4}, we have
$$|\ml_mf|=|\Pi_m\ml \Pi_mf|\le|\ml||f|\le M.$$
This implies the first statement of the lemma. To prove the second statement, we observe that
\begin{equation}\label{LY:Lip}
\begin{split}
\lip(\ml_mf)&=\lip(\Pi_m\ml \Pi_mf)\le\lip(\ml \Pi_mf)\\
&\le\alpha\lip(\Pi_mf)+B_1|\Pi_mf|\le\alpha\lip(f)+B_1|f|.
\end{split}
\end{equation}
Consequently, 
$$\lip(\ml_m^nf)\le\alpha^n\lip(f)+(1+\alpha+\cdots+\alpha^{n-1})M^{n-1}B_1|f|$$
and, for all $n\ge 1$,
$$||\ml_m^nf||\le\alpha^n||f||+C_3M^n|f|.$$
Finally, using Lemma \ref{Le4} and inequality (\ref{LY:Lip}), we obtain 
\begin{equation*}
\begin{split}
 |(\ml-\ml_m)f|&\le| (\ml-\Pi_m\ml)f|+|(\Pi_m\ml-\Pi_m\ml \Pi_m)f|\\
&\le2\varepsilon\lip(\ml f)+2\varepsilon M\lip(f)\le2\varepsilon[(\alpha+M)\lip(f)+B_1|f|] \le \Gamma\varepsilon||f||.
\end{split}
\end{equation*}
\end{proof}
\subsection{$\ml$ as a perturbation of $\ml_m$}
All the constants arising in Theorem \ref{Th1} are (in principle) computable for the
matrix representation of the finite-dimensional operator $\ml_m$, including the number $H_{\delta,r}(\ml_m)$. Thus, using the idea of \cite{Li}, we are going to apply Theorem \ref{Th1} with $\ml_m$ as $P_1$ and $\ml$ as the perturbation $P_2$. This entails some \emph{a priori} estimates.
\begin{lemma}\label{Le6}
Given $\ml$, $\delta>0$ and $r\in(\alpha,1)$, there exists $\varepsilon_3>0$ such that for each $\eta$ with
$0<\text{mesh}(\eta)\le\varepsilon_3$, we have
\begin{equation}\label{E5}
\text{mesh}(\eta)\le\Gamma^{-1}\varepsilon_0(\ml_m,r,\delta),
\end{equation}
and
\begin{equation}\label{E6}
|||\ml_m-\ml|||\le\varepsilon_0(P_{\eta},r,\delta).
\end{equation}
\end{lemma}
\begin{proof}
See Lemma 4.2 of \cite{Li}\footnote{Although the norms used in \cite{Li} are $BV$ and $L^1$, the proof of his Lemma 4.2 is valid in the general setting of \cite{KL} and does not really depend on the norms involved.}.
\end{proof}
In the next lemma we provide a computable upper bound on $H_{\delta,r}(\ml_m)$. 
\begin{lemma}\label{Le7}
For $f\in (\mathcal C, ||\cdot||)$ with $||f||=1$ and 
$z\in\mathbb C\setminus V_{\delta,r}(\ml_m)$ we have:
\begin{enumerate}
\item $||(z-\ml_m)^{-1}f||\le 
(\frac{B_1}{r-\alpha}+1)|(z-\ml_m)^{-1}\Pi_mf|+\frac{1}{r-\alpha}+\frac{2}{r}$;
\item $|(z-\ml_m)^{-1}\Pi_mf|\le ||(z-\ml_m)^{-1}f||+\frac{2}{r}$.
\end{enumerate}
\end{lemma}
\begin{proof}
We have
$$(z-\ml_m)^{-1}=z^{-1}(z-\ml_m)^{-1}\ml_m+z^{-1}\cdot{\bf 1}.$$
Then using the fact that $\Pi_m$ is a projection and statements (1) and (2) of Lemma \ref{Le3}, we obtain
\begin{equation}\label{st0}
\begin{split}
||(z-\ml_m)^{-1}(f-\Pi_mf)||&=||(z^{-1}(z-\ml_m)^{-1}\ml_m+z^{-1}\cdot{\bf 1})(f-\Pi_mf)||\\
&=||\frac{f}{z}-\frac{\Pi_m f}{z}||\le\frac{2}{|z|}||f||\le\frac{2}{r}.
\end{split}
\end{equation}
Now, write $(z-\ml_m)^{-1}\Pi_mf=h$, then $h=\frac{1}{z}(\ml_mh+\Pi_mf)$. Therefore, by inequality (\ref{LY:Lip}), we have
\begin{equation*}
\lip(h)\le\frac{1}{|z|}(\lip(\ml_mh)+1)\le\frac{1}{r}(\alpha \lip(h)+B_1||h||_1+1).
\end{equation*} 
Hence,
\begin{equation}\label{st1}
||h||\le(\frac{B_1}{r-\alpha}+1)|h|+\frac{1}{r-\alpha}.
\end{equation}
and by (\ref{st0}) and (\ref{st1}) the first part of the lemma follows. For the proof of the second statement, by (\ref{st0}), we have
\begin{equation}
\begin{split}
|(z-\ml_m)^{-1}\Pi_m f|&\le ||(z-\ml_m)^{-1}\Pi_m f||\le ||(z-\ml_m)^{-1}f||+\frac{2}{r}.
\end{split}
\end{equation}
\end{proof}
Now we define
$$H^*_{\delta,r}(\ml_m):=\sup\{(\frac{B_1}{r-\alpha}+1)|(z-\ml_m)^{-1}\Pi_m f|+\frac{1}{r-\alpha}+\frac{2}{r}:
 \| v\| _{1}=1, z\in\mathbb C\setminus V_{\delta, r}(\ml_m)\},$$
 $$n_2^*:=\lceil\frac{\ln 8C_3C_2(C_2+C_3+2)MH^*_{\delta,r}(\ml_m)}{\ln r/\alpha}+n_1\frac{\ln(M/r)}{\ln r/\alpha}\rceil$$ 
$$\varepsilon_1(P_1,r,\delta):= \frac{\left(r/M\right)^{n_1+n^*_2}}{8C_3(H^*_{\delta,r}(\ml_m)C_3+\frac{C_1}{M-r})},$$
\begin{equation*}
\begin{split}
&\varepsilon^*_2(\ml_m,r,\delta):=\\
&\left[\frac{1}{4C_3}\left(\frac{r}{M}\right)^{n_1}\left( \frac{1}{H^*_{\delta,r}(\ml_m)[C_2(C_2+C_3+2)M+C_3]+2C_2(C_2+C_3)+\frac{C_1}{M-r}}\right)\right]^{\frac{1}{\gamma}},
\end{split}
\end{equation*}
and
\begin{equation*}
\varepsilon_0^*(\ml_m,\delta,r):=\min\{\varepsilon^*_1(\ml_m,r,\delta), \varepsilon^*_2(\ml_m,r,\delta)\}
\end{equation*}
Note that $H^*_{\delta,r}(\ml_m)$ provides a computable upper bound on $H_{\delta,r}(\ml_m)$, and consequently $\varepsilon_0^*(\ml_m,\delta,r)$ provides a computable lower bound on $\varepsilon_0(\ml_m,\delta,r)$. 
The critical step is to obtain control on the separation of the point spectrum of $\ml$ outside the essential spectral radius $\alpha$. More precisely, the following algorithm will compute numbers $\delta = \delta_{\text{c}}$, $\alpha<r=r_{\text{c}}<1$ such that $\delta_{\text{c}}<1-r_{\text{c}}$ and $\varepsilon =\varepsilon_{\text{c}} > 0$ such that for any $\eta$ with $\text{mesh}(\eta)=\varepsilon_{\text{c}}$
\begin{enumerate}
\item $\text{mesh}(\eta)\le(\Gamma)^{-1}\varepsilon_0(\ml_m,r,\delta)$;
\item $B(\rho, \delta_{\text{c}})\cap B(0,r_{\text{c}})=\emptyset$, where $\rho$ is the dominant eigenvalue of $\ml_m$.
\item If $\rho_i\not=\rho$ is an eigenvalue of $\ml_m$, then $\rho_i\in B(0,r_{\text{c}})$.
\end{enumerate}
Thus we obtain the a spectral gap for $\ml_m$ and consequently, by Theorem \ref{Th1}, a spectral gap for $\ml$.
\begin{algorithm}\label{alg} Given $T$ and $p\in\mathbb N$, then
\begin{enumerate}
\item Pick $\delta=1/p$
\item Set $r=1-2\delta$.
\item Choose $m$, the number of partition intervals.
\item Set $\varepsilon = \frac{1}{m}$, the mesh size of the partition.
\item Find the matrix representation of $\ml_m$
\item Check if $\varepsilon\le(\Gamma)^{-1}\varepsilon_0^*(\ml_m,\delta,r)$.\\
If (6) is not satisfied, feed in a larger $m$ repeat (3)-(6); otherwise, continue.\\
\item Check that $B(\rho, \delta) \cap B(0,r)=\emptyset$ and $\rho_i\in B(0,r)$ for any eigenvalue $\rho_i$ of $\ml_m$ with $\rho_i\not=\rho$.
\item If (7) is satisfied, report $\delta_{\text{c}}:=\delta$, $r_{\text{c}}:=r$ and $m_{\text{c}}:=m$; otherwise, multiply $p$ by $2$ and repeat steps
(1)-(7) starting with the last $m$ that satisfied (7).
\end{enumerate}
\end{algorithm}
\begin{proposition}\label{A}
Algorithm \ref{alg} stops after finitely many steps.
\end{proposition}
\begin{proof}
The proof is similar to the proof of Proposition 1 of \cite{BB} .
\end{proof}
\subsection{A computable bound for the rate of decay of correlations of $\ml$}
We will benefit from Algorithm \ref{alg} in many directions. First it provides us with a nice spectral picture of ${\ml}$; i.e., an estimate on the size of its spectral gap; moreover it enables us to compute an upper bound on the norm of the resolvent of the continuous (infinite dimensional) operator $\ml$ bounded away from its spectrum. Such a computable estimate was impossible to do before Algorithm \ref{alg}.
\begin{lemma}\label{Le:res}
We have
\begin{equation*}
\begin{split}
\sup\{||(z-\ml)^{-1}||:\, z\in\mathbb C\setminus V_{\delta_{\text{c}},r_{\text{c}}}(\ml_{m_c})\}\le  \frac{4(C2+C3)(M/r_{\text{c}})^{n^*_1}}{(M-r_{\text{c}})2\varepsilon^*_1} :=H^*_{\delta_{\text{c}},r_{\text{c}}}(\ml),
\end{split}
\end{equation*}
where
$$\varepsilon^*_1= \frac{\left(r_{\text{c}}/M\right)^{n^*_1+n_2^*}}{8C_3(H^*_{\delta_{\text{c}},r_{\text{c}}}(\ml_{m_c})C_3+\frac{C_1}{M-r_{\text{c}}})},$$
$$n_1^*=\lceil\frac{\ln2C_2}{\ln r_{\text{c}}/\alpha}\rceil,$$
and
$$n_2^*=\lceil\frac{\ln 8C_3C_2(C_2+C_3+2)MH^*_{\delta_{\text{c}},r_{\text{c}}}(\ml_{m_{\text{c}}})}{\ln r_{\text{c}}/\alpha}+n^*_1\frac{\ln(M/r_{\text{c}})}{\ln r_{\text{c}}/\alpha}\rceil .$$
\end{lemma}
\begin{proof}
We have $\varepsilon_{\text{c}}\le\Gamma^{-1}\varepsilon_{0}(\ml_{m_{\text{c}}},\delta_{\text{c}},r_{\text{c}})$. Therefore, by Theorem \ref{Th1}, $z\in\mathbb C\setminus V_{\delta_{\text{c}},r_{\text{c}}}(\ml_{m_c})$,
$$||(z-\ml)^{-1}||\le \frac{4(C2+C3)(M/r)^{n_1}}{(M-r_{\text{c}})2\varepsilon_1(\ml_{m_{\text{c}}},\delta_{\text{c}}, r_{\text{c}})} \le \frac{4(C2+C3)(M/r)^{n_1}}{(M-r_{\text{c}})2\varepsilon^*_1}.$$
Here we have used the fact that $H_{\delta_{\text{c}},r_{\text{c}}}(\ml_{m_c})<H^*_{\delta_{\text{c}},r_{\text{c}}}(\ml_{m_c})$ which implies that $\varepsilon_1(\ml_{m_{\text{c}}},\delta_{\text{c}}, r_{\text{c}})>\varepsilon^*_1$. This completes the proof of the lemma.
\end{proof}
We now provide a computable estimate on the rate of decay of correlations of $\ml$.
\begin{lemma}\label{decay}
Set $R:=\min\{(1-\delta_{\text{c}}+r_{\text{c}})/2, (\rho_{\text{c}}-\delta_{\text{c}}+r_{\text{c}})/2\}$. Then for $f\in C^{\lip}(I)$, $\int f=0$, we have
$$||\ml^n f||\le R^{n+1}H^*_{\delta_{\text{c}},r_{\text{c}}}(\ml)||f||.$$
\end{lemma}
\begin{proof}
Let $\pi_{R}$ be the following spectral projection
$$\pi_{R}:=\frac{1}{2\pi i}\int_{\{z\in\mathbb C:\, |z|=R\}}(z-\ml)^{-1}dz.$$
Then for any $f\in C^{\lip}(I)$
$$\ml f=f^*\int f+\ml\pi_{R}f.$$
This implies that if $\int f=0$ we have $\pi_{R}f=f$ and consequently
$$\ml^n f=\frac{1}{2\pi i}\int_{\{z\in\mathbb C:\, |z|=R\}}z^n(z-\ml)^{-1}dz.$$
Thus, 
\begin{equation*}
\begin{split}
||\ml^n f||&\le R^{n+1}\sup_{|z|=R}||(z-\ml)^{-1}||\,||f||\\
&\le R^{n+1}\sup_{\{z\in\mathbb C\setminus V_{\delta_{\text{c}},r_{\text{c}}}(\ml_{m_c})\}}||(z-\ml)^{-1}||\,||f||\le R^{n+1}H^*_{\delta_{\text{c}},r_{\text{c}}}(\ml)||f||
\end{split}
\end{equation*}
\end{proof}


\section{Rigorous approximation with fast rates}\label{Rates}

The projection-based discretization does not appear to be particularly well suited 
to estimating the rate of approximation of the Perron eigenvector $f_{\eta_m}^*$ for $\ml_m$ to the invariant density $f^*$ for $\ml$, where $\eta_m$ denotes the 
uniform $m-$cell partition. We discuss this in a little more detail in the next section.
However,  we can use the spectral estimates on $\ml$ obtained via 
perturbations $\ml_m$ combined with the discretization scheme associated to 
$\mathbb{P}_m$ to obtain good, rigorous approximation rates in the uniform norm.  

We begin by deriving a suitable 
Lasota-Yorke inequality for the operator $\mathbb{P}_m$.
\begin{lemma}\label{Le:new1}
 For $f\in\mathcal C^{\lip}(I)$ we have $\mathbb{P}_m f\in\mathcal C^{\lip}(I)$ and
$$\lip(\mathbb{P}_m f)\le\beta\lip(f)+\bar B_1|f|,$$
where $\beta=(3/2)\alpha$ and $\bar B_1=(3/2)B_1$.
\end{lemma} 
\begin{proof}
The proof follows from Lemma \ref{Le2} and \ref{Le:new}
\end{proof}
Without loss of generality, we will assume\footnote{In fact we can choose $k_0\ge k$ so that that $M\alpha_0^k<2/3$ which makes $\beta<1$.} $\beta<1$. 

Let\footnote{Note the difference between $f^*_m$ and $f_{\eta_m}^*$ the Perron eigenvectors
for $\mathbb{P}_m$ and $\ml_m$ respectively.}  $f^*_m$ denote the nonnegative, normalized Perron eigenfunction of $\ml_m$ corresponding to the eigenvalue 1. Our goal 
is to prove $|f_m^* - f^*| = O(\frac{\ln m}{m})$

First, we find a computable upper bound on the Lipschitz norm of $f_{m}^*$. Due to the fact that 
$|f_{m}^*|\ge1$, this is best done in two steps.  
\begin{lemma}\label{lip:bd} 
$f^*_{m}$ and $f^*$ satisfiy the following estimates
\begin{enumerate}
\item $|f^*_m|\le K_1$,
\item $|f^*|\le K_2$,
\item $\lip (f^*_m)\le \frac{\bar B_1}{(1-\beta)}K_1$,
\item $\lip (f^*)\le \frac{B_1}{1-\alpha}K_2$
\end{enumerate}
where $K_1:=\frac{\bar B_1}{1-\beta}+1$ and $K_2:=( \frac{B_1}{1-\alpha})+ 1$.
\end{lemma}
\begin{proof}
For the first estimate, note that the matrix representing $\mathbb{P}_m$ (as determined in Section
\ref{sec:markov}) is stochastic and strictly positive, hence mixing. It follows that 
$$f^*_m = \lim_n \mathbb{P}_m^n \mathbf{1}$$
uniformly on the interval $I$. But,as a consequence of Lemma \ref{Le:new1} we get 
$$\lip(\mathbb{P}_m^n \mathbf{1})| \leq \beta^n \lip(\mathbf{1}) + \frac{\bar B_1}{1-\beta} |\mathbf{1}|
= \frac{\bar B_1}{1 - \beta}.$$ Since $\mathbb{P}_m^n \mathbf{1}$ is 
nonnegative, with unit $L^1-$ norm we find 
$| \mathbb{P}_m^n \mathbf{1}| \leq 1 +  \frac{\bar B_1}{1 - \beta}$.  This is the first estimate.

The second estimate is similar, using Lemma \ref{Le1} in place of Lemma \ref{Le:new1}
to obtain a uniform bound  on the Lipschitz constants and $L^\infty-$norm of the iterates:
$\lip( \ml^n \mathbf{1}) \leq  \frac{B_1}{1-\alpha}$ and $|  \ml^n \mathbf{1}| \leq \frac{B_1}{1-\alpha} +1$.
By Helley's Theorem, we may obtain a uniformly convergent subsequence from the sequence of iterates; the 
averages along this subsequence converge uniformly to $f^*$, giving the second estimate. 

Next, using Lemma \ref{Le:new1}, we obtain
$$\lip(f_{m}^*)=\lip(\mathbb{P}_m f_{m}^*)\le\beta\lip(f_{m}^*)+\bar B_1|f_{m}^*|.$$
Therefore, 
$$\lip(f_{m}^*)\le \frac{\bar B_1}{(1-\beta)}|f_{m}^*|\leq \frac{\bar B_1}{(1-\beta)}K_1.$$ 
The same argument, using Lemma \ref{Le1} and estimate (2) gives the inequality in part (4).
\end{proof}
Now we state and proof of our main result, Theorem \ref{Th2}. There are certainly many ways to proceed which lead to a convergence rate of order $\frac{\ln(m)}{m}$ \footnote{ To improve the rate of convergence from $O(\frac{\ln m}{m})$ to $O(\frac{1}{m})$, one would need $\lip(f^*-Q_{m}f^*)=O(\frac{1}{m})$. However, this cannot be obtained unless $f^*$ is very regular. For the class of maps under consideration $f^*$ is, in general, only Lipschitz continuous and hence $\lip(f^*-Q_mf^*)=O(1)$.
The situation is analogous to the failure to obtain rate $O(\frac{1}{m})$ for Ulam's method
via spectral perturbation arguments; see \cite{DL}.}. However, different proofs will lead to different constants multiplied by this rate. Since our target is to produce explicit computation of the error bound, we write the proof in a way which makes the constants as small as possible.
\begin{theorem}\label{Th2}
For any $m\in\mathbb N$, we have
$$|f^*-f^*_m|\le \frac{(K_1^2 +K_2^2)RH^*_{\delta_{\text{c}},r_{\text{c}}}(\ml)+ \frac{5M^2}{2}\lceil\ln(m)/\ln R^{-1}\rceil K_1}{m}.$$
\end{theorem}
\begin{proof}
We have
\begin{equation}\label{dens1}
\begin{split}
|f^*-f^*_{m}|&\le |\mathbb{P}_m^n f^*_{m}-\ml^{n}f^*_{m}|+|\ml^{n}f^*_{m}-\ml^{n}f^*|\\
&= (I) + (II).
\end{split}
\end{equation}
Using Lemmas  \ref{decay} and \ref{lip:bd}, for any $n$ we obtain
\begin{equation}\label{dens2}
 (II) \le(K_1(1 + \frac{\bar B_1}{1-\beta})  +K_2(1 + \frac{B_1}{1-\alpha}))H^*_{\delta_{\text{c}},r_{\text{c}}}(\ml) R^{n+1}.
\end{equation} 
In $(I)$ the problem is that we only have the weak estimate $|\mathbb{P}_m^n|\le (2M)^n$. However, we know from Lemma \ref{Le1} that for all $n\ge 1$,  $|\ml^n|\le M$. Therefore we change the order in (I) and benefit from Lemmas \ref{Le1} and \ref{lip:bd}. 
So, we write
\begin{equation}\label{dens3}
\begin{split}
(I)&= |\ml^{n}f^*_{m}-\mathbb{P}_m^nf^*_{m}|\le\sum_{q=1}^{n}|\ml^{n-q}(\ml-\mathbb{P}_m)\mathbb{P}_m^{q-1}f^*_{m}|\\
&\le M\sum_{q=1}^{n}|(\ml-\mathbb{P}_m)(\mathbb{P}_m^{q-1}f^*_{m})|= M\sum_{q=1}^{n}|(I-Q_m)
(\ml \mathbb{P}_m^{q-1}f^*_{m})|\\
&\le\frac{5M^2}{2m}\sum_{q=1}^{n}|f^*_{m}|\le  \frac{5M^2}{2m}nK_1 .
\end{split}
\end{equation}
Choosing $n=\lceil\ln(m)/\ln(R^{-1})\rceil$ and using (\ref{dens2}) and (\ref{dens3}) completes the proof.
\end{proof}
\section{Example}\label{Exam}
In this example we use the map:
 \begin{equation}\label{map}
\tau(x)=\left\{\begin{array}{cc}
\frac{11x}{1-x}&\mbox{for $0\le x\le \frac{1}{12}$}\\
12x-i&\mbox{for $\frac{i}{12}<x\le\frac{i+1}{12}$}
\end{array}
\right. ,
\end{equation}
where $i=1,\dots,11$. 
\subsection{Computing $H^*_{\delta_{\text{c}},r_{\text{c}}}(\ml)$}
We will use Equation \ref{eq:variation_estimate} to estimate the relevant constants.
For $f\in BV$ we have 
$$V\ml_{\tau}f\le1/11Vf+2/11\|f\|_1.$$
Consequently, for any $n\ge 1$, we obtain
$$V\ml^n_{\tau}f\le(1/11)^nVf+1/5\|f\|_1.$$
Hence $\alpha_0=1/11$, $B_0=1/5$, and $M=6/5$. We set $T:=\tau$. Then $\alpha=\alpha_0M=6/55$. For $f\in\mathcal C^{\lip}(I)$ and any $n\ge1$
$$\lip(\ml f)\le6/55\lip (f)+2/(11)^2|f|.$$
We now use Algorithm \ref{alg} to obtain $m_{\text{c}}$ which allows us to compute an upper bound on the norm of the resolvent of the continuous operator $H^*_{\delta_{\text{c}},r_{\text{c}}}(\ml)$\footnote{For details on how to use the computer to compute an upper bound on $H^{*}_{\delta,r}(\ml_m)$ in the algorithm for the discrete operator $\ml_m$ see \cite{BB}.}.
\begin{table}[h]\label{table1}\caption{The output of Algorithm \ref{alg} for the map in (\ref{map})}
\begin{center}
\begin{tabular}{|c|c|c|}
  \hline
  $r$ & 0.8 & 0.8\\
  \hline
  $\delta$ & 0.1 & 0.1\\
  \hline
  $\varepsilon$ &$2\times 10^{-4}$ & $8\times10^{-5}$\\
  \hline
  $H^{*}_{\delta,r}(\ml_m)$ &  27.2822974 & 27.29122985\\
    \hline
    $(\Gamma)^{-1}\varepsilon_*$& 0.0001276212886&0.0001275730002\\
   \hline
$\text{Loop I}$& \text{Fail: reduce $\varepsilon$} &\text{Pass}\\
\hline
\text{Loop II}& \text{}& \text{Pass}\\
\hline
\text{Output}& \text{}& $\text{mesh}(m_{\text{c}}):=8\times10^{-5}$, $\delta_{\text{c}}:=0.1$, $r_{\text{c}}:=0.8$\\
\hline

\end{tabular}
\end{center}
 \end{table}

 Using this output, by Lemma \ref{Le:res}, we obtain
$$H^*_{\delta_{\text{c}},r_{\text{c}}}(\ml):=\frac{4(C2+C3)(M/r_{\text{c}})^{n_1}}{(M-r_{\text{c}})2\varepsilon^*_1}\le 36366.11326.$$
\subsection{Computing a uniform bound on the approximation error}
The first few eigenvalues of $\ml$ are estimated (to 4 decimal places) by
$$1.0000,~0.0901,~ 0.0088          
~0.0000 + 0.0069i,
~0.0000 - 0.0069i
$$ 
Next we compute the constants from Lemma \ref{decay} and  Lemma \ref{lip:bd}:
$$K_1=\frac{521}{506},\, K_2=\frac{549}{539} \text{ and } R= 0.8500$$                              
Then in Table \ref{table2} we list the number of bins $m$ and the corresponding approximation error of $|f^*-f^*_m|$. 
\begin{table}[h]\label{table2}\caption{ Error $:=|f^*-f^*_m|$, vs {m} where $f^*$ is the invariant density of the map in (\ref{map}) and $f^*_m$ is the invariant density of $\ml_m$}
\begin{center}
\begin{tabular}{|c|c|}
  \hline
  $m$ & Error\\
  \hline
  $10^6$ &0.06515863735\\
  \hline
  $7\times10^6$ &0.009314201609\\
    \hline
    \end{tabular}
 
\end{center}
 \end{table}
The results in Table \ref{table2} are undoubtebly far from optimal.  We reiterate that efficiency in computations has not 
been our main focus.

 \section{Discussion on the rate of convergence of the discretization schemes $\ml_m$}\label{discuss}
 The operator $\ml_m$ is not  Markov and its dominant eigenvalue $\rho$ is typically $> 1$. This hinders an estimate similar to that of $(I)$ in Theorem \ref{Th2}. In particular the sum $\sum_{q=1}^{n}\ml_m^q f^*_{\eta_m}$ cannot be controlled since $\ml_m^q f^*_{\eta_m}=\rho^q f^*_{\eta_m}$. Thus 
 the key trick for obtaining rate of convergence $\varepsilon\ln\varepsilon^{-1}$ 
 is unavailable for the scheme $\Pi_m$ (unless it happens that $\rho \leq 1$). 
 
 In principle, one can obtain rigorous, computational estimates for the scheme $\Pi_m$, through 
 direct application of Theorem \ref{Th1},  obtaining $L^{\infty}$-norm estimates
 in terms of the difference of the spectral projections
$$|\proj_1-\proj_{\rho}|\le\delta_{\text{c}}[aH^*+b{H^*}^2]\varepsilon^{\gamma_{\text{c}}},$$
where $\gamma_{\text{c}}=\frac{\ln (r_{\text{c}}/\alpha)}{\ln(M/\alpha)}<1$. From 
this estimate, it follows that for the map in (\ref{map}),  to achieve an error smaller than $10^{-2}$, one would need $\varepsilon<e^{-35.9}$. 
Needless to say, this is not a practical approach in general.

On the other hand, one expects that 
the both discretization schemes should be 
more efficient than indicated by the above theoretical 
estimates. We conclude this article with a simple numerical example (non-rigorous) that 
shows the kind of performance one should be looking for, even in the non-Markov 
discretization case. 

\begin{figure}[htbp] 
   \centering
   \includegraphics[width=4in]{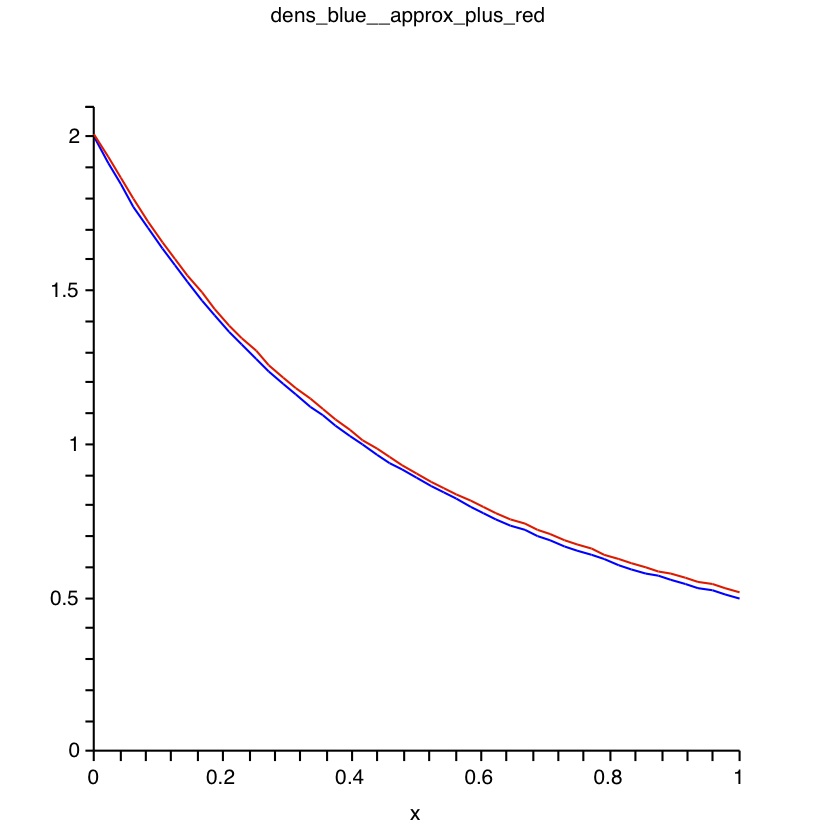} 
   \caption{The graphs of $f^*$ (blue) and $f_{10}^*+0.02$} (red) for $\Pi_m-$ scheme
   \label{fig:1}
\end{figure}

Consider the following Markov interval map which is only eventually expanding:
 \begin{equation}\label{map:ev}
\tau(x)=\left\{\begin{array}{cc}
\frac{2x}{1-x}&\mbox{for $0\le x\le \frac{1}{3}$}\\
\frac{1-x}{2x}&\mbox{for $\frac{1}{3}<x\le 1$}
\end{array}
\right. ,
\end{equation}
and whose invariant density is given by
$$f^*(x)=\frac{2}{(1+x)^2}.$$
Since the invariant density $f^*$ is known for this example, we can check the efficiency of our discretization scheme $\Pi_m$.  Using $m=10$, we plotted in Figures \ref{fig:1} and \ref{fig:2} the graphs of $f^*$, $f^*_{\eta_m}+0.02$ and $f^*-f_{\eta_m}^*$ respectively. As the figures show, our 
tolerance $\Delta$ is attained with only 10 bins; the efficiency of the discretization scheme $\Pi_m$ in achieving an $L^{\infty}$ approximation is pretty impressive.

 \begin{figure}[htbp] 
   \centering
   \includegraphics[width=4in]{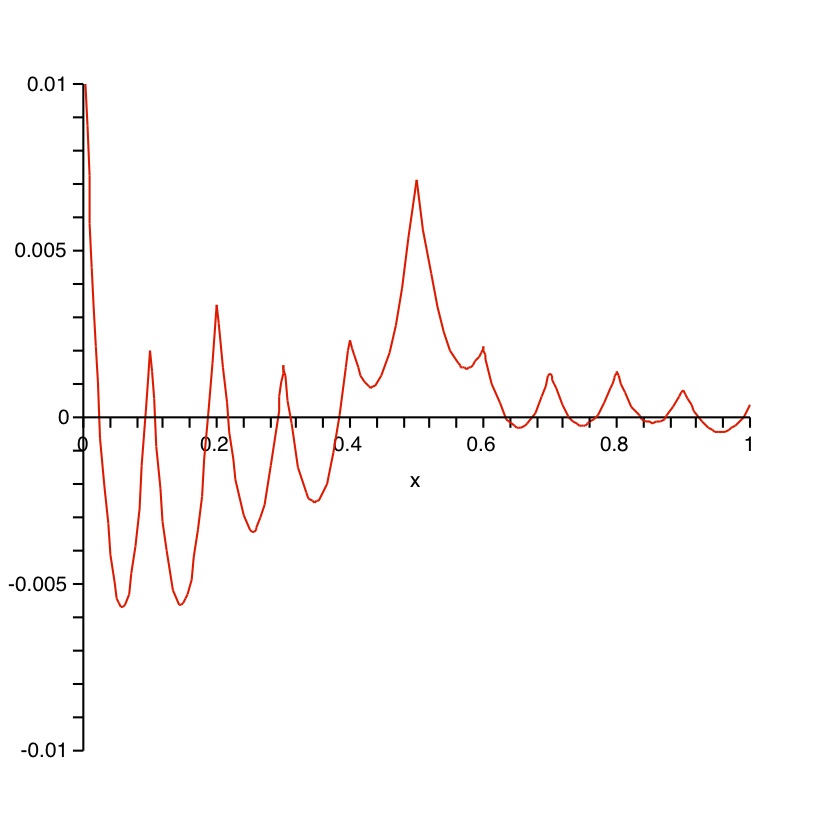} 
   \caption{The graph of $f^*-f_{10}^*$} for the $\Pi_m-$ scheme
   \label{fig:2}
\end{figure}

\bibliographystyle{amsplain}

\end{document}